\documentclass[12pt]{amsart}

\usepackage[usenames]{color}
\input xypic

\usepackage{verbatim}
\usepackage{url}
\usepackage{bm}

\usepackage{amsmath}
\usepackage[centertags]{amsmath}
\usepackage{amsfonts}
\usepackage{amssymb}
\usepackage{amsthm}
\usepackage{newlfont}
\usepackage{url}
\usepackage{bm}

\theoremstyle{plain}
\newtheorem{thm}{Theorem}
\newtheorem{cor}[thm]{Corollary}

\newtheorem{lem}[thm]{Lemma}
\newtheorem{lem*}[thm]{Lemma}
\newtheorem{prop}[thm]{Proposition}
\newtheorem{question}[thm]{Question}

\newtheorem{ass}[thm]{Assumptions}

\theoremstyle{definition}
\newtheorem{dfn}{Definition}
\theoremstyle{remark}
\newtheorem{rem}{Remark}
\newtheorem{rem*}{Remark}

\numberwithin{rem}{section} 
\numberwithin{dfn}{section} 
\numberwithin{equation}{section} 
\numberwithin{thm}{section} 

\def\hb{\hfil\break}

\def\!{\operatorname{!}}

\def\1{\bold 1}

\def\Aut{\operatorname{Aut}}

\def\deg{\operatorname{deg}}

\def\ord{\operatorname{ord}}

\def\Spec{\operatorname{Spec}}

\def\Gal{\operatorname{Gal}}

\def\Hom{\operatorname{Hom}}
\def\GL{\operatorname{GL}}

\def\Im{\operatorname{Im}}

\def\tr{\operatorname{tr}}

\def\deg{\operatorname{deg}}

\def\id{\operatorname{id}}
\def\Id{\operatorname{Id}}

\def\End{\operatorname{End}}
\def\pr{\operatorname{pr}}

\def\mod{\operatorname{mod}}

\def\tor{\operatorname{tor}}

\begin{document}

\title{On a reduction map for Drinfeld modules}

\author{Wojciech Bondarewicz, Piotr Kraso\'n}
\begin{abstract}
In this paper we investigate a local to global principle for Mordell-Weil group defined over a ring of integers ${\cal O}_K$  of $\mathbf t$-modules that are products of the Drinfeld modules ${\widehat\varphi}={\phi}_{1}^{e_1}\times \dots \times {\phi}_{t}^{e_{t}}.$  
Here $K$  is a finite extension of the field of fractions of $A={\mathbb F}_{q}[t].$
We assume that the ${\text{\rm rank}(\phi}_{i})=d_{i}$ and  endomorphism rings of the involved Drinfeld modules of generic characteristic are the simplest possible, i.e. ${\End}_{K^{{\mathrm{sep}}}}({\phi}_{i})=A$ for $ i=1,\dots , t.$ 
Our main result is the following numeric criterion. Let   ${N}={N}_{1}^{e_1}\times\dots\times {N}_{t}^{e_t}$ be a finitely generated $A$ submodule of the Mordell-Weil group ${\widehat\varphi}({\cal O}_{K})={\phi}_{1}({\cal O}_{K})^{e_{1}}\times\dots\times {\phi}_{t}({\cal O}_{K})^{{e}_{t}},$ and let ${\Lambda}\subset N$ be an A - submodule.  If we assume $d_{i}\geq e_{i}$ and  $P\in N$ such that ${\mathrm{red}}_{\cal W}(P)\in {\mathrm{red}}_{\cal W}({\Lambda}) $ for almost all primes ${\cal W}$ of ${\cal O}_{K},$
then $P\in {\Lambda}+N_{tor}.$ We also build on the recent results of S.Bara{\'n}czuk \cite{b17} concerning the dynamical local to global principle in Mordell-Weil type groups and the solvability of certain dynamical equations  to the aforementioned ${\mathbf t}$-modules.
\end{abstract}
\date{\today}

\address{University of Szczecin, ul. Wielkopolska 15, 70-451 Szczecin, Poland \hb
\indent University of Szczecin, ul. Wielkopolska 15, 70-451 Szczecin, Poland}
\email{wbondarewicz@gmail.com, piotrkras26@gmail.com}

\thanks{}

\maketitle

\section{Introduction} 
The local to global principle we investigate is of the following type. Assume we are given an object  $A$ (of a small category $\cal C$) associated with a ring of integers ${\cal O}_R$ of some ring $R$. Assume that for any 
prime ideal $\frak p \lhd R$ there exists a reduced object $A_{\frak p}. $ Assume further that we are given a property $PROP$ of $A $ such that the corresponding property $PROP_{\frak p}$ for $A_{\frak p}$ makes sense.
Then one can raise the following question:
\begin{question}[Local to global principle]\label{question}
Assume that for almost all prime ideals ${\frak p}\lhd R$  properties $PROP_{\frak p}$ hold. Does it follow that property $PROP$ hold for $A$?
\end{question}
In 1975 A. Schinzel \cite{sch75}, generalized the work of Skolem \cite{sk37} from 1937 and proved the following theorem concerning exponential equations.
\begin{thm}[A.Schinzel]\label{tw1}
If ${\alpha}_{1},\dots {\alpha}_{k}, \beta $ are non-zero elements of $K$ and the congruence 
$${\alpha}_{1}^{x_1}{\alpha}_{2}^{x_2}\dots {\alpha}_{k}^{x_k}\equiv \beta \,\, mod \,{\mathfrak p}$$ 
is soluble for almost all prime ideals $\mathfrak p$ of a number field $K$ then the corresponding 
equation is soluble in rational integers. 
That is there exist $n_1\dots \
n_k \in {\mathbb Z}$ such that $\beta = {\alpha_{1}}^{n_1}\dots {\alpha_{k}}^{n_k}.$
\end{thm}
This theorem is in fact the {\it detecting linear dependence} problem for number fields and of the kind considered in Question \ref{question}.
The reduction map is the usual reduction modulo a non Archimedean prime in a number field.
It is well known that some questions concerning number fields can be translated to the context of abelian varieties. 
An analogous question for abelian varieties was raised by W. Gajda (see \cite{we03} ) and has been studied extensively \cite{ko03},\cite{we03},\cite{bgk05},\cite{gg09},\cite{jo13}. 
A related problem, called {\it the support problem},  originated from a question by P.Erd{\"os} for integers and was solved in \cite{crs97} for number fields and elliptic curves. 
The support problem for some abelian varieties and intermediate Jacobians was treated in \cite{bgk03} and solved for all abelian varieties by M.Larsen in \cite{la03}.
In the case of Drinfeld modules, the support problem was studied by A.Li in \cite{li06}.
For the history  of the detecting linear dependence problem for abelian varieties, as well as generalisation of it to linear algebraic groups see \cite{fk17}.

 
In \cite{bk11} the second named author and G.Banaszak proved the following theorem:

\begin{thm}\label{thm A}
 Let ${\mathbb A}/F$ be an abelian variety
defined over a number field $F.$ Assume that
${\mathbb A}$ is isogenous to ${\mathbb A}_{1}^{e_1} \times \dots \times {\mathbb A}_{t}^{e_{t}}$
with ${\mathbb A}_i$ simple, pairwise nonisogenous abelian varieties
such that  $${dim_{{\mathrm{End}}_{F^{\prime}}({\mathbb A}_i)^{0}}\,
H_{1} ({\mathbb A}_i({\mathbb C});\, {\mathbb Q})\, \geq e_i}$$
for each $1 \leq i \leq t,$ 
where ${\mathrm{End}}_{F^{\prime}}({\mathbb A}_{i})^{0} := {\mathrm{End}}_{F^{\prime}}({\mathbb A}_{i})\otimes {\mathbb Q}$
and $F^{\prime}/F$ is a finite extension
such that the isogeny is defined over $F^{\prime}.$ 
Let $P \in {\mathbb A}(F)$ and let
$\Lambda$ be a subgroup of ${\mathbb A}(F).$ 
If ${\mathrm{red}}_v (P) \in {\mathrm{red}}_v (\Lambda)$ for almost all
$v$ of  ${\cal O}_F$ then $P \in \Lambda + {\mathbb A}(F)_{\mathrm{tor}}.$
\noindent
Moreover if ${\mathbb A}(F)_{{\mathrm{tor}}}
\subset \Lambda,$
then the following conditions are equivalent: 
\begin{itemize}
\item[1]\,\, $P \in \Lambda$
\item[2]\,\, ${\mathrm{red}}_v (P) \in {\mathrm{red}}_v (\Lambda)$ for almost all
$v$ of  ${\cal O}_F.$
\end{itemize}
\end{thm}
Theorem \ref{thm A} gives a numeric criterion needed for a local to global principle to hold (up to torsion). Let us make  a few more comments on Theorem \ref{thm A}. 
This principle for abelian varieties  with a commutative endomorphism ring was proven in \cite{we03}. Notice that the reduction map makes sense, since for an abelian variety ${\mathbb A}$ over a number field one has the N{\'e}ron model 
${\cal A}$ such that the Mordell-Weil group ${\cal A}({\cal O}_{F})={\mathbb A}(F)$ (cf. \cite{blr90}). Notice also that for abelian varieties over global fields, one has the Poincar{\'e} decomposition theorem, i.e.
any abelian variety ${\mathbb A}/F$ can be decomposed over some field $F^{\prime} \supset F$ uniquely up to an isogeny  as  a product ${\mathbb A}={\mathbb A}_{1}^{e_{1}}\times \dots \times {\mathbb A}_{t}^{e_{t}}$ where ${\mathbb A}_{i}, \, i=1,\dots , t$ are geometrically simple abelian varieties 
cf. \cite{m70}.

The main result of this current paper is the result corresponding to Theorem \ref{thm A} for Drinfeld modules, or rather Anderson ${\mathbf t}$-modules \cite{a86}, \cite{bp09} that are products of Drinfeld modules. The general framework of the proof is similar
to that of \cite{bk11} or \cite{bk13}. ( In \cite{bk13} the local to global principle for {\'e}tale $K$-theory of curves was treated.) As in previous papers, there are significant subtle differences between the proofs.
Let $A={\mathbb F}_q[t]$ be the ring of polynomials of one variable over a finite field ${\mathbb F}_{q}.$
Our main theorem is the following:

\begin{thm}\label{localglobal0}
Let ${\widehat\varphi}={\phi}_{1}^{e_{1}}\times\dots\times {\phi}_{t}^{e_{t}}$ be a ${\mathbf t}$-module where ${\phi}_{i}$ for $1\leq i \leq t$ are pairwise non-isogenous Drinfeld modules of generic characteristic  defined over ${\cal O}_K.$ 
 We assume that for each  $1\leq i\leq t,$
${\End}_{K^{\mathrm{sep}}}({\phi}_{i})=A.$ Let $N_{i}\subset {\phi}_{i}({\cal O}_{K})$ be a finitely generated $A$-submodule of the Mordell-Weil group. Pick ${\Lambda}\subset N=N_{1}^{e_{1}}\times \dots \times N_{k}^{e_{t}}$ to be an $A$-submodule. Assume that $d_{i}={{\mathrm rank}}({\phi}_{i})\geq e_{i}$ for each $1\leq i\leq t.$ Let $P\in N$ and assume that for almost all primes ${\cal P}$ of ${\cal O}_K$ we have
${\mathrm{red}}_{\cal P}(P)\in {\mathrm{red}}_{\cal P}({\Lambda}).$ Then $P\in {\Lambda}+N_{\tor}.$
\end{thm}

\begin{rem}\label{remarkO}
Let ${\End}_{K^{\mathrm{sep}}}({\phi}_{i})=A$
and $S_{i}$ denote the finite set of places of $K$ at which ${\phi}_i$ (understood as defined over $K$) has bad reduction (cf. Definition \ref{badp}) . Let $S={\cup}_{i=1}^t S_{i}.$ Notice that our proofs work for the $S$-integers 
${\cal O}_{K,S}$ i.e we can assume in Theorem \ref{localglobal} that the Drinfeld modules  ${\phi}_i, \, 1\leq i\leq t $ are defined over ${\cal O}_{K,S}$ (cf. Definition \ref{def3}).
\end{rem}
\begin{rem}\label{remark1}
Notice also that for any Drinfeld module ${\phi}: A \rightarrow K\{\tau \}$  defined over $K$ with the endomorphism ring $A$  there exists a minimal model over ${\Spec O}_{K,S_{\mathrm{bad}}},$ where  $S_{\mathrm{bad}}$ is the finite set 
of primes of bad reduction cf.  \cite{t93} Proposition 2.2.
\end{rem}

The basic definitions concerning Drinfeld modules and  reduction maps are given in the following sections. Let us  emphasize  the major differences here between the situations described in Theorems \ref{thm A} and \ref{localglobal0}.
Firstly, the category of Drinfeld modules is not semisimple, therefore we have to state our theorem for ${\mathbf t}$-modules that are products of Drinfeld modules.   Secondly, the Mordell-Weil group of a Drinfeld module of generic characteristic is not finitely generated, thus we have 
 to be careful to use  a finitely generated $A$-submodule of the Mordell-Weil group.  Finally, our proof relies on the {\em reduction theorem} below. In the proof of Theorem \ref{redthm0}, we use the Ribet-Bashmakov method \cite{r79}, developed  
 for Drinfeld modules in \cite{p16}, \cite{h11}. This method works nicely for Drinfeld modules ${\phi}$ for which ${\End}_{K^{\mathrm{sep}}}({\phi})=A.$

\begin{thm}\label{redthm0}
Let $A={\mathbb F}_q[t]$, ${\phi}_i$ for $1\leq i \leq t$ be  Drinfeld modules of generic characteristic defined over $K$ such that ${\End}_{K^{\mathrm{sep}}}({\phi}_{i})=A$, ${\cal P}\in M_{A}$ be a maximal ideal and ${\pi}_P$ its generator.
Let $x_{i,j}\in {\phi}_{i}({\cal O}_K)$ for some $s_{i}$ and $1\leq j \leq s_{i}$ be linearly independent elements over $A$ for each $1\leq i\leq t.$ There is a  set $W$ of prime ideals $\cal W$ of ${\cal O}_{K}$ of positive density such that  ${\mathrm{red}}_{\cal W}(x_{i,j})=0$ in $ {\phi}_{i}^{{\mathcal W}}({\mathcal O}_K/{{\cal W}})_{{\pi}_{\cal P}}$ for all $1\leq j\leq s_{i}$ and  $1\leq i\leq t.$
\end{thm}
As a corollary we obtain the following:
\begin{thm}\label{redthm1}
Let $A={\mathbb F}_q[t]$, ${\phi}_i$ for $1\leq i \leq t$ be  Drinfeld modules of generic characteristic defined over $K$ such that ${\End}_{K^{\mathrm{sep}}}({\phi}_{i})=A.$
Let  ${\cal P}\in M_{A},$  $m\in {\mathbb N}\cup \{0\}$.
Let $x_{i,j}\in {\phi}_{i}({\cal O}_K)$ for some $s_{i}$ and  $1\leq j \leq s_{i}$ be linearly independent elements over $A$  and let $T_{i,j}\in {\phi}_{i}[{\cal P}^m]$  be arbitrary torsion points for all $1\leq j\leq s_{i}$ and $1\leq i\leq t.$     
Then there is a set $W$ of prime ideals $\cal W$ of ${\cal O}_{K}$  of positive density such that  
$${\mathrm{red}}_{{\cal W}^{\prime}}(T_{i,j})={\mathrm{red}}_{{\cal W}}(x_{i,j})$$ in $ {\phi}_{i}^{{\mathcal W}}({\mathcal O}_K/{{\cal W}})_{{\pi}_{\cal P}}$ for all $1\leq j\leq s_{i}$ and  $1\leq i\leq t,$ where 
${\cal W}^{\prime}$ is a prime in $O_L$ over ${\cal W},$ where $L$  is the compositum   of fields ${K({\phi}_{i}[{\cal P}^m])}\, {\mathrm{for}} \, 1\leq i\leq t$, ${\mathrm{red}}_{{\cal W}^{\prime}}:{\phi}_{i}({\cal O_L})\rightarrow {\phi}_{i}^{\cal W}({k_{{\cal W}^{\prime}}})$ and ${k_{{\cal W}^{\prime}}}={{{\cal O}_L}}/{{\cal W}^{\prime}}.$
\end{thm}

\begin{rem}\label{rma}
Here ${\phi}_{i}^{{\mathcal W}}({\mathcal O}_K/{{\cal W}})_{{\pi}_{\cal P}} = \{ {\alpha}\in {\phi}_{i}^{{\cal W}} ({\mathcal O}_K/{{\cal W}}) |\quad \exists  \, k\in {\mathbb N}\,\, {\text{\rm such that}} \,\ {\pi}_{\cal P}^k{\alpha}=0  \,\}     $ 
and the equality ${\mathrm{red}}_{\cal W}(x_{i,j})=0$  (resp. ${\mathrm{red}}_{{\cal W}^{\prime}}(T_{i,j})={\mathrm{red}}_{{\cal W}}(x_{i,j})$ ) in $ {\phi}_{i}^{{\mathcal W}}({\mathcal O}_K/{{\cal W}})_{{\pi}_{\cal P}}$ means, by slight abuse of notation that it holds after projection ${\phi}_{i}^{{\mathcal W}}({\mathcal O}_K/{{\cal W}})\rightarrow {\phi}_{i}^{{\mathcal W}}({\mathcal O}_K/{{\cal W}})_{{\pi}_{\cal P}}.$
\end{rem}

We view Theorems \ref{redthm0} and \ref{redthm1} as interesting on their own, not only as one of the key steps in proving Theorem \ref{localglobal0}. Some other applications are described in Section 6, where we show that the recent results of S.Bara{\'n}czuk can be extended to  the situation we consider.
The content of the paper is as follows. In Section 2 we review some general definitions and facts concerning Drinfeld modules. The reader is advised to consult the  general sources \cite{g96},\cite{l96},\cite{bp09}.
Section 3 is devoted to Kummer theory \cite{p16}, \cite{h11}. In Section 4 we prove the reduction theorems, i.e. Theorems \ref{redthm0} and \ref{redthm1}. 
In Section 5 we give a proof of the local to global principle for ${\mathbf t}$-modules that are products of Drinfeld modules. In Section 6 we state theorems analogous to S.Bara{\'n}czuk's  and indicate how to prove them.

We  also  would  like to thank the referee for numerous valuable remarks and suggestions which helped us to improve the manuscript.

\section{Preliminaries on Drinfeld modules}
Let ${\mathbb F}_{q}$ be a finite field with $q=p^{m}$ elements. Let $F$ be a field of transcendence degree $1$ over ${\mathbb F}_{q},$  i.e. a function field of a smooth projective curve $X$ over ${\mathbb F}_{q}$ and let $A$ be the ring of elements of $F$ regular outside a fixed closed point 
$\infty$. Let $K$ be a finitely generated field over ${\mathbb F}_{q}$. The ring of ${\mathbb F}_{q}$-linear endomorphisms ${\End}_{{\mathbb F}_{q}}({\mathbb G}_{{a},K})$ of the additive algebraic group over $K$ is the twisted  (non-commutative) polynomial ring in one variable $K\{\tau\}.$ The endomorphism $\tau$ corresponds to $u\rightarrow u^q$ and the commutation relation is ${\tau}u=u^q\tau , \, u\in K.$ 

\begin{rem}
We prove our main theorems for the case $F={\mathbb F}_{q}(t)$  and $A={\mathbb F}_{q}[t],$ so in our case $X={\mathbb P}^{1}_{{\mathbb F}_{q}}.$
\end{rem}

\begin{dfn}\label{def1}
An $A$-field $K$ is a fixed morphism  
 ${\iota}: A \rightarrow K$. The kernel of $\iota$ is a prime ideal $\cal P$ of $A$ called the characteristic. The characteristic of $\iota$ is called finite if ${\cal P}\neq 0,$ or generic (zero) if ${\cal P}= 0.$
\end{dfn}
\begin{dfn}\label{def2}
A Drinfeld $A$-module is a homomorphism $\phi : A\rightarrow K\{\tau\},\,\, a\rightarrow {\phi}_{a}\,,$ of ${\mathbb F}_{q}$-algebras such that 
\begin{enumerate}
\item[1.] $D\circ {\phi}={\iota},$
\item[2.] for some $a\in A, \, {\phi}_{a}\neq {\iota}(a){\tau}^{0},$
\end{enumerate}
where $D({\sum}_{i=0}^{i={\nu}}\,\,a_{i}{\tau}^{i})=a_{0}. $ The characteristic of a Drinfeld module is the characteristic of ${\iota}.$
\end{dfn}
\begin{dfn}\label{def3}
We say that ${\phi}$ is defined over ${\cal O}_K$ or ${\phi}$ has integral coefficients  iff  \linebreak ${\phi}: A\rightarrow {\cal O}_K\{\tau\}.$ 
Similarly, if $S$ is a finite set of places of ${\cal O}_K$ and ${\cal O}_{K,S}$ denotes the ring of $S$-integers in ${\cal O}_K$ we say that ${\phi}$ is defined over ${\cal O}_{K,S}$  iff 
${\phi}: A\rightarrow {\cal O}_{K,S}\{\tau\}.$
\end{dfn}
Let $\phi$ be a Drinfeld module over the $A$-field $K.$
Let
\begin{equation*}
{\mu}_{\phi}(a):=-\deg {\phi}_{a}(\tau),\quad {\mu}_{\phi}(0)=-{\infty}.
\end{equation*}
It is easy to prove (\cite{g96} Lemma 4.5.1 ) that  ${\mu}_{\phi}(a)=-d \deg(a)$ for any $a.$ The integer $d$ is called the rank of the Drinfeld module ${\phi}.$
Assume that $K$ is an $A$-field with a non-trivial discrete valuation $v$ and $v(A)\ge 0. $ Let ${\cal O}_{v}=\{{\gamma}\in K \, | \, \, v(\gamma)\ge 0\}$ be the valuation ring of $K.$  
Let ${\phi}$ be a Drinfeld module with integral coefficients, i.e. every ${\phi}_{a}\in {\cal O}_K\{\tau\}.$  There exists a reduction ${\phi}^v$ of a Drinfeld module defined over $k_v={\cal O}_v/{\frak{m}}_{v}$ (cf. \cite{g96} Definition 4.10.1).  By \cite{g96} Lemma 4.10.2, for any Drinfeld module defined over $K$ there exists a Drinfeld module with integral at $v$ coefficients isogenous to it. For the definition of an isogeny of Drinfeld modules see \cite{g96} Definition 4.4.3. 

\begin{dfn}\label{def5}
Let ${\phi}$ be a Drinfeld module over $K,$  and $L$ be an algebraic extension of $K.$ The Mordell-Weil group ${\phi}(L)$ is the additive group of $L$ viewed as an $A$-module via evaluation of the polynomials ${\phi}_{a},\,\, a\in A.$
\end{dfn}
The rank of an $A$-module $M$ is the dimension of the $F$-vector space $M{\otimes}_{A}F.$ An $A$-module is called tame if all its submodules  of finite rank are finitely generated.
B. Poonen (\cite{p95} Theorem 1) proved that ${\phi}({L})$ is the direct sum of a finite torsion submodule and a free $A$-module of rank ${\aleph}_{0}.$ 
However, by \cite{p95} Lemma 4, ${\phi}(L)$ is a tame $A$-module.

Let $I\subseteq A$ be an ideal. In general, for any $A$,  $I$ is generated by two elements $\{a_{i_{1}}, a_{i_{2}}\}.$ Let ${\phi}_{I}$ be a monic polynomial   which is a right greatest common divisor of  ${\phi}_{a_{i_{1}}}$ and 
${\phi}_{a_{i_{2}}}.$  It exists since in $K\{\tau\}$ one has a right division algorithm ( cf. \cite{g96}). ${\phi}_{I}$ is a generator of the left ideal ( in $K\{\tau\}$ ) generated by ${\phi}_{a_{i_{1}}}$ and 
${\phi}_{a_{i_{2}}}$ (\cite{g96} Definition 4.4.4). Let $\bar K$ denote an algebraic  closure of $K.$
\begin{dfn}\label{def6}
For an ideal $I$ let ${\phi}[I]\subset {\phi}(\bar K)$ be the finite subgroup of roots of ${\phi}_{I}.$
\end{dfn}
Notice that since $I$ is an ideal of $A$,  ${\phi}[I]$ is stable under $\{{\phi}_{a}\},\,\, a\in A$ and the Galois group $G_{K}=\Gal (K^{\text{\rm sep}}/K),$ of the separable closure $K^{\text{\rm sep}}\subset {\bar K},$ acts on ${\phi}[I].$
For any ideal $I$ prime to characteristic we have (cf. \cite{ro02} Corollary to Theorem 13.1, \cite{g96} )
\begin{equation*}\label{rediso}
{\phi}[I]\cong (A/I)^{d}
\end{equation*}
and thus a representation
\begin{equation*}\label{redrep}
{\bar\rho}_{I}: G_{K} \rightarrow {\Aut}_{A}({\phi}[I])\cong {\GL}_{d}(A/I).
\end{equation*}
\begin{dfn}\label{def6}
Let ${\phi}$ be Drinfeld module and $\cal P$ be a maximal ideal different from the characteristic ${\cal P}_{0}.$ The ${\cal P}$-adic Tate module is defined as 
\begin{equation}\label{eq1}
T_{\cal P}({\phi})={\Hom}_{A}(F_{\cal P}/A_{\cal P}, {\phi}[{\cal P}^{\infty}])
\end{equation}
where 
${\phi}[{\cal P}^{\infty}]={\bigcup}_{m\ge 1}\,\,{\phi}[{\cal P}^m]$
and $F_{\cal P}$ (resp. $A_{\cal P}$) is the ${\cal P}$-adic completion of $F$ (resp. $A$).
\end{dfn}
Let ${\phi}^{{n}} : {\phi}[{\cal P}^{n+1}]\rightarrow {\phi}[{\cal P}^{n}] $ be  the multiplication by ${\pi}_P$ map.
Then (\ref{eq1}) can be written in the following way
\begin{equation}\label{eq2}
T_{\cal P}({\phi})\cong \varprojlim  {\phi}[{\cal P}^m], \quad  .
\end{equation}
Notice that  (\ref{eq2}) becomes a free $A_{\cal P}$-module of rank $d$
and $G_{K}$ acts on $T_{\cal P}({\phi})$ continuously. Since the action  of $G_{K}$ commutes with the multiplication by elements of $A_{\cal P}$ we obtain a ${\cal P}$-adic representation
\begin{equation}\label{padicrep}
{\rho}_{\cal P} : \, G_{K}\rightarrow {\Aut}_{A_{\cal P}}(T_{\cal P}({\phi}))\cong {\GL}_{d}(A_{\cal P}).
\end{equation}
Let ${\mathrm{red}}_{\cal P}: {\GL}_{d}(A_{\cal P})\rightarrow  {\text{\rm GL}}_{d}({k}_{\cal P})$ be the projection map. Then we have ${\bar\rho}_{\cal P}={\mathrm{red}}_{\cal P}\circ {\rho}_{\cal P}.$

Let $M_{A}$ be the set of all maximal ideals of $A$ and let $
{\hat A}={\prod}_{{\cal P}\in M_{A}} A_{{\cal P}}.$
Under the assumption that $\phi$ is of generic characteristic and ${\End}_{K^{\mathrm{sep}}}(\phi)=A$ in \cite{pr109}, the authors proved that the adelic representation
\begin{equation}\label{openimage}
{\rho}_{\text{\rm ad}}: G_{K} \rightarrow {\GL}_{d}({\hat A})
\end{equation}
has open image.

For our purposes, especially for the proof of Theorem \ref{best}, we need the following general result  ( \cite{wa01} Proposition 6) concerning Mordell-Weil groups of  Drinfeld modules defined  over {\em finitely generated}  ( over the field of fractions $F$ of $A$)     fields $L.$
\begin{prop}[\cite{wa01}]\label{wang}
Let $\bar L$  (resp.   $L^{\mathrm sep}$)  be an algebraic ( resp. separable ) closure of $L.$
Each of the $A$-modules ${\phi}(\bar L)$ and ${\phi}(L^{\mathrm sep})$ is the direct sum of a $F$-vector space of dimension ${\aleph}_0$ with a torsion submodule. Furthermore, when the $A$-characteristic is generic, the torsion submodule of each case is isomorphic to $(F/A)^{d}$ .When the $A$-characteristic of $L$ is ${\cal P}$,the torsion submodule of ${\phi}(\bar L)$ (resp. ${\phi}(L^{\mathrm sep})$ ) is isomorphic to ${\oplus}_{{\beta}\neq {\cal P}} (F_{\beta}/A_{\beta})^{d} \oplus (F_{\cal P}/A_{\cal P})^{d-{\bar h}} $ (resp. a submodule between ${\oplus}_{{\beta}\neq {\cal P} }(F_{\beta}/A_{\beta})^{d} $ and ${\oplus}_{{\beta}\neq {\cal P}} (F_{\beta}/A_{\beta})^{d} \oplus (F_{\cal P}/A_{\cal P})^{d-{\bar{h}}} )$ where $\beta$ is taken in ${\Spec}A$, $d$ and ${{\bar h}}$ are the rank and the height of ${\phi}$ respectively.

\end{prop}

\section{Kummer theory}
From now on we assume that $A={\mathbb F}_{q}[t]$ is the  ring of polynomials in one variable and $F={\mathbb F}_{q}(t)$ is the field of rational functions over ${\mathbb F}_q.$
In order to prove  Theorem \ref{redthm0} we need Kummer theory in the context of Drinfeld modules, which corresponds to one constructed  by Ribet  \cite{r79} for the extensions of abelian varieties by tori.
Such an extension was developed in \cite{h11} and \cite{p16}. Now we recall  from these references the relevant facts below. For a detailed exposition see the original sources.

Consider a Drinfeld module ${\phi}: A\rightarrow K\{\tau\}$ of generic characteristic. We will consider $K^{\text{\rm sep}}$ as an $A$-module via ${\phi}.$ We assume that ${\End}_{K^{\mathrm{sep}}}(\phi)=A.$
Let $s\ge1$ and ${\Lambda}$ be an $A$-submodule of $K$ generated by $s$  $A$-linearly independent elements $x_{1},\dots ,x_{s}.$
Let $y_{i}\in {\phi}^{-1}_{I}(\{x_{i}\})\subset K^{\text{\rm sep}}.$ This is possible since ${\phi}_{I}-x_{i}$ is a separable polynomial.
Define 
\begin{equation}\label{eta}
{\eta}: G_{K} \rightarrow {\phi}[I] ^s\cong {\text{\rm Mat}}_{d\times s}(A/I), \quad {\eta}(\sigma)=({\sigma}(y_1)-y_{1},\dots , {\sigma}(y_s)-y_{s})
\end{equation}

and 
\begin{equation}\label{PhiI}
{\Phi}_{I}: G_{K} \rightarrow  {\text{\rm Mat}}_{d\times s}(A/I)\rtimes {\GL}_{d}(A/I), \quad {\Phi}_{I}(\sigma)=({\eta}(\sigma) , {\bar{\rho}}_{I}(\sigma)).
\end{equation}

The action of ${\GL}_{d}$ on ${\text{\rm Mat}}_{d\times s}$ is given by matrix multiplication.
     Now,  following \cite{h11}, we describe the map analogous to (\ref{PhiI})  for $A_{\cal P}, \,{\cal P}\in M_{A}.$

Let
${{\phi}[{{\cal P}^{\infty}}]= {\bigcup}_{n\geq 0}{\phi}[{\cal P}^{n}]\subset K^{\text{\rm sep}}, {\phi}[{\cal P}^{0}]:=0},$
$K_{{\cal P}^{\infty}}=K({\phi}[{{\cal P}^{\infty}}]).$ Let $K_{{\text{\rm ad}}}$ be the \linebreak compositum of the fields $K_{{\cal P}^{\infty}}$ for all ${{\cal P}\in M_{A}}.$ 
and ${\pi}_{{\cal P}}^{n}$ be a generator of ${\cal P}^n.$ For an $A$-submodule $M$ of $K^{\text{\rm sep}}$ and $n\geq 0$ denote as 
${\phi}_{{\pi}_{{{\cal P}}}^{n}}^{-1}(M)\subset K^{\text{\rm sep}}$
the inverse image of $M$ under the endomorphism ${\phi}_{{\pi}_{{{\cal P}}}^{n}}$ (cf. Definition \ref{def2} ).
By a slight abuse of notation, denote  by ${\phi}^{n}$  (cf. (\ref{eq2} )) the map
${\phi}^{n}: {\phi}_{{\pi}_{{{\cal P}}}^{n+1}}^{-1}(M)\rightarrow  {\phi}_{{\pi}_{{\cal P}}^{n}}^{-1}(M)$ as well.
The extended ${\cal P}$-adic Tate module is defined as
$T_{\cal P}[M]=\varprojlim  {\phi}^{-1}_{{\pi}_{\cal P}^{n}}(M).$
Let $K\subset L\subset K^{sep}$  and $M\subset L$ then the absolute Galois group $G_{L}$ acts trivially on $M$ and acts continuously on $T_{\cal P}[M]$ by the following formula:
$
{\sigma}(t_{n})=({\sigma}(t_{n})), \, (t_{n})\in T_{\cal P}[M].
$
One has an exact sequence of $G_{L}$-$A$-modules
\begin{equation}\label{es}
0\rightarrow T_{\cal P}(\phi) \rightarrow T_{\cal P}[M]\rightarrow M\rightarrow 0
\end{equation}
where the surjection in (\ref{es}) is given by the map ${\pr}_{M}: T_{\cal P}[M]\rightarrow M,\quad {\pr}_{M}((t_{n})_{n\geq 0}) =t_{0}.$ 
For any ${\sigma}\in G_{L}$ and $(t_{n})\in T_{\cal P}[M]$ one has ${\sigma}(t_{n})-t_{n}\in {\ker}({\phi}_{{\pi}_{\cal P}^{n}})$ and therefore 
\begin{equation}\label{ksi1}
{\phi}^{n}({\sigma}(t_{n+1})-t_{n+1} ) = {\sigma}({\phi}^{n}(t_{n+1}))- {\phi}^{n}(t_{n+1})={\sigma}(t_{n})-t_{n}.
\end{equation}
Formulas (\ref{ksi1}) show that the map
$$
{\xi}_{L,M}:T_{\cal P}[M]\rightarrow {\text{\rm Map}}(G_{L},T_{\cal P}(\phi))\qquad (t_{n})\rightarrow [{\sigma} \rightarrow ({\sigma}(t_n)-t_{n})_{n\geq 0}]
$$
is well defined.
One can specialize this construction to $M={\Lambda}$ and $K=L.$ For each $1\leq i\leq s$ choose $(t_{i,n})\in {\pr}_{\Lambda}^{-1}(x_{i})$ and  construct the map
\begin{equation}\label{eta}
{\eta}_{\cal P}: G_{K} \rightarrow (T_{\cal P}({\phi}))^{s} \cong {\text{\rm Mat}}_{d\times s}(A_{\cal P})
\end{equation}
by the formula:
${\eta}_{\cal P}(\sigma)= ({\xi}_{K,\Lambda}((t_{1,n})_{n\geq 0})(\sigma),\dots , {\xi}_{K,\Lambda}((t_{s,n})_{n\geq 0})(\sigma)).$
Using  \cite{h11} Lemma 4.3  one obtains the $A$-module homomorphism 
$${\bar\xi}_{K_{\text{\rm ad}},{\Lambda}} :{\Lambda}\rightarrow {\Hom}(G_{K_{\text{\rm ad}}}, T_{\cal P}(\phi)) \qquad (t_{n})_{n\geq 0} \rightarrow[{\sigma}\rightarrow ({\sigma}(t_{n}) - t_{n})_{n\geq 0}].$$
Let ${\psi}_{i}={\bar\xi}_{K_{\text{\rm ad}},{\Lambda}}(x_{i})\in  {\Hom}(G_{K_{\text{\rm ad}}}, T_{\cal P}(\phi))$ for the fixed generators $x_{1},\dots , x_{s}$ of $\Lambda$ and let
\begin{equation}\label{eqq3}
{\Psi}_{\cal P} : G_{K_{\text{\rm ad}}} \rightarrow (T_{\cal P}(\phi))^{s}\cong {\text{\rm Mat}}_{d\times s}(A_{\cal P}), \qquad {\sigma}\rightarrow ({\psi}_1(\sigma),\dots , {\psi}_{s}(\sigma)).
\end{equation}
For $T\subset M_{A}$ there is the homomorphism
$${\Psi}_{T}: G_{K_{\text{\rm ad}}} \rightarrow {\prod}_{{\cal P}\in T}\, {\text{\rm Mat}}_{d\times s}(A_{\cal P}),\qquad {\sigma}\rightarrow ({\Psi}_{\cal P}(\sigma))_{{\cal P}\in T}.$$
Denote ${\Psi}_{\text{\rm ad}}:={\Psi}_{M_{A}}.$
The main results in \cite{h11} are the following:
\begin{prop}[ \cite{h11} Proposition 4.6]\label{prop1}
The image of ${\Psi}_{\cal P}$ is equal to ${\mathrm{ Mat}}_{d\times s}(A_{\cal P})$ for almost all 
${\cal P} \in M_{A}$ and is open for all ${\cal P} \in M_{A}.$
\end{prop}
\begin{thm}[ \cite{h11} Theorem 4.4]\label{mainthm}
The image of ${\Psi}_{{\text{\rm ad}}}$ is open.
\end{thm}

In the sequel for  prime ideals ${\cal P}\in M_{A}$ we shall use the following modules:
$$V_{{\cal P}}(\phi)=T_{{\cal P}}({\phi}){\otimes}_{{A}_{\cal P}}F_{\cal P}.$$
\begin{rem}\label{dim}
Notice that in view of (\ref{eqq3}) we have $T_{\cal P}(\phi)\cong {{\text{\rm Mat}}}_{d\times 1}(A_{\cal P})$ and ${\dim}_{F_{\cal P}}V_{\cal P}(\phi)=d.$
\end{rem}

\section{Proof of the reduction theorem}
In what follows we assume that the fields of definitions for the Drinfeld modules involved  are  finite extension of $F={\mathbb F}_{q}(t).$

Let ${\phi}:A \rightarrow {\cal O}_K\{\tau \}$ be a Drinfeld module of generic characteristic defined over $O_K$.    For a maximal ideal ${{\cal P}}\subset {\mathcal O}_K$ we can reduce coefficients of ${\phi}_{a},$ for every $a\in A,$ modulo ${{\cal P}}$ and obtain a Drinfeld module over a finite field ${\mathcal O}_K/{{\cal P}}.$
This Drinfeld module has the special characteristic ${\cal P}.$ We will denote it by ${\phi}^{{\mathcal P}}.$
\begin{dfn}\label{badp}
Let ${\phi}$ be a fixed Drinfeld module of rank $d$ defined over $K.$ A prime ${\cal P}\in {\Spec} {\cal O}_K$ is called {\em good} if
\begin{enumerate}
\item[(1)] there exists ${\alpha}\in K^{\times}$ such that ${\phi}_{1}={\alpha}{\phi}{\alpha}^{-1}$ has ${\cal P}$-integral coefficients.
\item[(2)] the reduced map ${{\phi}_{1}^{\cal P}} : A \rightarrow {\cal O}_{K}/{\cal P}\{\tau \}$ is a Drinfeld module of rank $d$.
\end{enumerate}
Primes that are not good are called {\em bad}.
\end{dfn}

It is well known (because $A$ is a finitely generated ring over ${\mathbb F}_{q})$ that almost all primes ${\cal P}\in {\Spec}{\cal O}_{K}$ are good. Moreover, for almost all primes one can take ${\alpha}=1$  (cf. \cite{g94}, p.320, \cite{th04}).
There also exists  an analogue of classical {\em N{\'e}ron - Ogg - ${\check{S}}$afarevi${\check{c}}$} criterion for abelian varieties \cite{ta82}, \cite{g94}.

\begin{dfn}\label{red}
By reduction map mod ${{\cal P}}$  for Mordell-Weil groups we will mean the $A$-module homomorphism 
$$
{\mathrm{ red}}_{{\mathcal P}}: {\phi}({\mathcal O}_K) \rightarrow  {\phi}^{{\mathcal P}}({\mathcal O}_K/{{\cal P}}).
$$
\end{dfn}

Let $L$ be a finite extension of $K.$

The following Proposition holds true  ( \cite{bko95} Proposition 1.3):
\begin{prop}\label{injection}
Let ${\phi}$ be a Drinfeld module over ${\cal O}_{L}$ of rank $d,$  ${\cal P}$ be a maximal ideal of ${\cal O}_{L}$ and let ${\frak p}=Ker( A\rightarrow {\cal O}_{L} \rightarrow {\cal O}_{L}/{\cal P})$ be the special characteristic. Let $I$ be prime to ${\frak p}.$ Then
\begin{enumerate}
\item[(1)]  $Tor_{\phi}({\cal P})=\{ x\in\cal P: {\phi}_{a}(x)=0 $ for some $ a\in A \}$ has no nontrivial $I$-torsion.
\item [(2)] the reduction map is an injection on the $I$-torsion  between ${\phi}({\cal O}_{L})[I]$ and  ${\phi}^{\cal P}(O_{L}/{\cal P})[I].$
\end{enumerate}
\end{prop}
\begin{rem}
In \cite{bko95} the authors consider Drinfeld modules of rank 2 but their proof works for arbitrary rank.
\end{rem}

\begin{thm}\label{redthm}
Let $A={\mathbb F}_q[t]$, ${\phi}_i$ for $1\leq i \leq t$ be  Drinfeld modules of generic characteristic defined over $K$ such that ${\End}_{K^{\mathrm{sep}}}({\phi}_{i})=A$, ${\cal P}\in M_{A}$ be a maximal ideal and ${\pi}_P$ its generator.
Let $x_{i,j}\in {\phi}_{i}({\cal O}_K)$ for some $s_{i}$ and $1\leq j \leq s_{i}$ be linearly independent elements over $A$ for each $1\leq i\leq t.$ There is a  set $W$ of prime ideals $\cal W$ of ${\cal O}_{K}$ of positive density such that  ${\mathrm{red}}_{\cal W}(x_{i,j})=0$ in $ {\phi}_{i}^{{\mathcal W}}({\mathcal O}_K/{{\cal W}})_{{\pi}_{\cal P}}$ for all $1\leq j\leq s_{i}$ and  $1\leq i\leq t.$
\end{thm}

\begin{proof}
By Proposition \ref{prop1} there exists $m\in {\mathbb N}$ such that for any $\cal P$ one has
$$
{\cal P}^{m}{\prod}_{i=1}^{t} T_{\mathcal P}({\phi}_{i})^{s_{i}} \subset {\Psi}_{\cal P} (G_{K_{\mathrm{ad}}}) \subset {\prod}_{i=1}^{t}T_{\mathcal P}({\phi}_{i})^{s_{i}}.
$$
Let ${\Gamma}={\sum}_{i=1}^{t}{\sum}_{j=1}^{s_{i}}Ax_{i,j},$ denote  $K_{{\cal P}^{\infty}}:= K({\widehat\varphi}[{\cal P}^{\infty}]),$  where ${\widehat\varphi}[{\cal P}^{\infty}] = {\bigcup}_{i=1}^{t} {\phi}_{i}[{\cal P}^{\infty}]$ and let  
$$\frac{1}{{\pi}^{\infty}_{\cal P}}{\Gamma}:=\{x\in K^{\text{\rm sep}}\, | \,\,  \, \exists \,m \,\, {\text{\rm such that}}\,\, {\pi}^m_{\cal P}\,x\in {\Gamma}\, \}.$$
Similarly, define $K_{{\cal P}^k}:=K({\widehat\varphi}[{\cal P}^k])$ and 
$$\frac{1}{{\pi}^{k}_{\cal P}}{\Gamma}:=\{x\in K^{\text{\rm sep}}\, | \, \,\, {\text{\rm such that}}\,\, {\pi}^{k}_{\cal P}\,x\in {\Gamma}\, \}.$$
Let $H_{{\cal P}^{k}}=G(K^{\mathrm{sep}}/K_{{\cal P}^{k}})$ and $H_{{\cal P}^{\infty}}=G(K^{\mathrm{sep}}/K_{{\cal P}^{\infty}}).$
The Kummer map (\ref{eta})  yields the maps
\begin{equation}\label{psiij}
{\psi}_{i,j}^{(k)} : H_{{\cal P}^{k}} \rightarrow {\phi}_i[{\cal P}^{k}], \quad {\psi}_{i,j}^{(k)}(\sigma):= {\sigma}(\frac{1}{{\pi}_{\cal P}^k}x_{i,j})- \frac{1}{{{\pi}_{\cal P}}^k}x_{i,j}
\end{equation} 
Let ${\psi}_{i,j}$ be the inverse limit of $ {\psi}_{i,j}^{(k)}.$ 
 In the following commutative diagram (diagram (\ref{diag1}))  $\overline{{\Psi}^k_{\cal P}}=  \bigoplus_{i=1}^t \bigoplus_{j=1}^{s_{i}}(\overline{\psi}_{i ,j}^{(k)}) $  and $\overline{{\Psi}_{\cal P}}=  \bigoplus_{i=1}^t \bigoplus_{j=1}^{s_{i}}(\overline{\psi}_{i ,j}),$  where 
$\overline{\psi}_{i ,j}^{(k)}$ and $\overline{\psi}_{i ,j}$ are the maps induced by ${\psi}_{i ,j}^{(k)}$ and ${\psi}_{i ,j}.$ 
Notice that by  construction the maps $\overline{\psi}_{i ,j}^{(k)}$ and ${\overline{{\Psi}_{\cal P}^k}}$ are injective for all $k\geq 1.$
\begin{equation}\label{diag1}
 \xymatrix{
G(K_{{\cal P}^{\infty}}(\frac{1}{{\pi}^{\infty}_{\cal P}} \Gamma)/K_{{\cal P}^{\infty}}) \ar[d] \ar[r]^-{\overline{{\Psi}_{\cal P}}} &  \bigoplus_{i=1}^t T_{{\cal P}} ({\phi}_i)^{s_i}/{\cal P}^{m}
\bigoplus_{i=1}^t T_{{\cal P}} ({\phi}_i)^{s_i} \ar[d]^{\cong} \\
G(K_{{\cal P}^{k+1}}(\frac{1}{{\pi}^{k+1}_{\cal P}} \Gamma)/K_{{\cal P}^{k+1}})
 \ar[d] \ar[r]^-{\overline{\Psi^{k+1}_{\cal P}}} &  \bigoplus_{i=1}^t ({\phi}_i [{\cal P}^{k+1}])^{s_i}/
{\cal P}^{m} \bigoplus_{i=1}^t ({\phi}_i [{\cal P}^{k+1}])^{s_i} \ar[d]^{\cong}\\
G(K_{{\cal P}^{k}}(\frac{1}{{\pi}^{k}_{\cal P}} \Gamma)/K_{{\cal P}^{k}})
 \ar[r]^-{\overline{{\Psi}_{\cal P}^k}} &  \bigoplus_{i=1}^t ({\phi}_i [{\cal P}^{k}])^{s_i} /
{\cal P}^{m} \bigoplus_{i=1}^t ({\phi}_i [{\cal P}^{k}])^{s_i}
}
\end{equation}
Since the maps ${\overline{{\Psi}_{\cal P}^k}}$ and ${\overline{{\Psi}_{\cal P}^{k+1}}}$ are injective and the right bottom vertical arrow is an isomorphism we see that the left bottom arrow is an injection. 
This shows  the following equality for orders of the Galois groups $|G(K_{{\cal P}^{k+1}}(\frac{1}{{\pi}^{k+1}_{\cal P}} \Gamma)/K_{{\cal P}^{k+1}})|\leq |G(K_{{\cal P}^{k}}(\frac{1}{{\pi}^{k}_{\cal P}} \Gamma)/K_{{\cal P}^{k}})|.$
As these Galois groups are finite the equality is achieved for sufficiently large $k.$ Therefore the left bottom vertical arrow is also  a surjection and 
 the images of ${\overline{{\Psi}_{\cal P}^{k+1}}}$ and 
${\overline{{\Psi}_{\cal P}^k}}$ in the diagram (\ref{diag1}) are isomorphic for large enough $k.$ 
Let us consider the following diagram of fields
\begin{equation}\label{diag2}
\diagram
&&&K_{{\cal P}^{k+1}}(\frac{1}{{\pi}_{\cal P}^{k+1}}{\Gamma})\dline&\\
&&&K_{{\cal P}^{k+1}}(\frac{1}{{\pi}_{\cal P}^{k}}{\Gamma})&\\
&& K_{{\cal P}^{k}}(\frac{1}{{\pi}_{\cal P}^{k}}{\Gamma}) \urline&& K_{{\cal P}^{k+1}} \ulline \\
&&& K_{{\cal P}^{k}}\ulline\urline&&\\
\enddiagram
\end{equation}
The above mentioned surjection of Galois groups yields the following equality for large enough values of $k.$ 
\begin{equation}\label{intersection}
K_{{\cal P}^{k}}(\frac{1}{{\pi}_{\cal P}^{k}}{\Gamma}) \cap K_{{\cal P}^{k+1}}=K_{{\cal P}^{k}}.
\end{equation}
Let ${\tilde h}\in G(K_{{\cal P}^{\infty}}/K_{{\cal P}^k})$ be the homothety $1+{\pi}_{\cal P}^{k}u, u\in{\mathbb F}_{q}^{*},$ acting on $T_{\cal P}(\widehat{\phi})$. For $k\gg 0$ such a homothety exists according to the open image theorem 
of Pink and R{\"u}tsche \cite{pr09} (cf. ( \ref{openimage}))  and the fact that $K_{{\cal P}^k}$ is a finite extension of $K.$ Let $h\in  G(K_{{\cal P}^{k+1}}/K_{{\cal P}^k})$ be a projection of ${\tilde h}.$
By (\ref{intersection}) there exists ${\sigma}\in G(K_{{\cal P}^{k+1}}(\frac{1}{{\pi}_{\cal P}^k}\Gamma)/K_{{\cal P}^k})$ such that ${\sigma}|_{K_{{\cal P}^{k}}(\frac{1}{{\pi}_{\cal P}^k}\Gamma)}=\id$ and 
${\sigma}|_{K_{{\cal P}^{k+1}}}=h.$
By the Chebotarev density theorem for global fields (cf. \cite{fj08}, Theorem 6.3.1) there is a  set of primes ${\cal W} \in {\cal O}_K$ of positive density such that there exists a prime ${\cal W}_{1} \in {\cal O}_{K_{{\cal P}^{k+1}}(\frac{1}{{\pi}_{\cal P}^{k+1}}{\Gamma})}$ with the Frobenius in $G(K_{{\cal P}^{k+1}}(\frac{1}{{\pi}_{\cal P}^{k+1}}{\Gamma}))$ equal to ${\sigma}.$ Assume that ${\cal W}\nmid {\cal P}.$
Let ${\cal W}$ and ${\cal W}_{1}$ be such primes and let ${\cal W}_{2}\in {\cal O}_{K_{{\cal P}^{k}}(\frac{1}{{\pi}_{\cal P}^{k}}{\Gamma})}$ be the prime below ${\cal W}_{1}.$
Consider the following commutative diagram:

\begin{align}\label{diag3}
\xymatrix{
{\phi}_i ({\cal O}_{K})\ar[d]  \ar[r]^-{\,\,\quad{\text{\rm red}}_{\cal W}} &  {\phi}_{i}^{{\cal W}} (k_{{\cal W}})_{{\pi}_{\cal P}} \ar[d]^{=}\\
{\phi}_{i}({\cal O}_{K_{{\cal P}^k}(\frac{1}{{\pi}_{\cal P}^{k}}{\Gamma})})
\ar[d] \ar[r]^{\,\,\quad{\text{\rm red}}_{{\cal W}_2}} & {\phi}_{i}^{{\cal W}_2} (k_{{\cal W}_2})_{{\pi}_{\cal P}} \ar[d]\\ 
{\phi}_i ({\cal O}_{K_{{\cal P}^{k{+}1}}(\frac{1}{{\pi}_{\cal P}^{k}}{\Gamma})})
\ar[r]^{\,\,\quad {\text{\rm red}}_{{\cal W}_1}} &  {\phi}_{i}^{{\cal W}_1} (k_{{\cal W}_1})_{{\pi}_{\cal P}}
}
\end{align}
where for brevity we denoted $k_{\cal W}:={\cal O}_{K}/{\cal W}$ and similarly for $k_{{\cal W}_{1}}$ and $k_{{\cal W}_{2}}.$ The subscript ${\pi}_{\cal P}$ for the  Drinfeld modules with finite coefficients denotes the 
${\pi}_{\cal P}$-torsion e.g. $ {\phi}_{i}^{{\cal W}} (k_{{\cal W}})_{{\pi}_{\cal P}} = \{ {\alpha}\in {\phi}_{i}^{{\cal W}} (k_{{\cal W}}) |\quad \exists  \, k\in {\mathbb N}\,\, {\text{\rm such that}} \,\ {\pi}_{\cal P}^k{\alpha}=0  \,\}.     $ 
In the diagram \ref{diag3}  reduction maps are by slight abuse of notation the compositions  of the reduction maps with the projections on the ${\pi}_{\cal P}$-torsion part (cf. Remark \ref{rma} ).
\begin{dfn}\label{piorder}
By the ${\pi}_{\cal P}$ -order of a point $x\in {\phi}_{i}^{{\cal W}} (k_{{\cal W}})_{{\pi}_{\cal P}}$ we mean the  least  positive integer $m$ such that ${\pi}_{\cal P}^mx=0.$
\end{dfn}
\begin{rem}\label{remarkw}
Notice that by our choice of ${\cal W}, {\cal W}_{1}$ and ${\cal W}_2$  the ${\cal W}_{1}$-Frobenius  element ${\sigma}$ restricts to the identity on $K_{{\cal P}^{k}}(\frac{1}{{\pi}^{k}}{\Gamma}).$   
Therefore we have $k_{{\cal W}_2}=k_{\cal W}$ and ${\phi}_{i}^{{\cal W}_2} (k_{{\cal W}_2})_{{\pi}_{\cal P}}={\phi}_{i}^{{\cal W}} (k_{{\cal W}})_{{\pi}_{\cal P}}.$
\end{rem}

Let $c_{i,j}$ be the ${\pi}_{\cal P}$-order of ${\text{\rm red}}_{\cal W}(x_{i,j})\in {\phi}_{i}^{{\cal W}} (k_{{\cal W}})_{{\pi}_{\cal P}}.$ All the vertical arrows in the diagram (\ref{diag3}) are injections.
Let $y_{i,j}=\frac{1}{{\pi}_{\cal P}^k}x_{i,j} \in {\phi}_{i}({\cal O}_{K_{{\cal P}^k}(\frac{1}{{\pi}_{\cal P}^{k}}{\Gamma})}) \subset {\phi}_i ({\cal O}_{K_{{\cal P}^{k{+}1}}(\frac{1}{{\pi}_{\cal P}^{k}}{\Gamma})}) .$ 
One readily verifies that the ${\pi}_{\cal P}$-order of ${\text{\rm red}}_{{\cal W}_1} (y_{i,j})$ equals $k+c_{i,j}.$ By Remark \ref{remarkw} the element  ${\text{\rm red}}_{{\cal W}_1}(y_{i,j})$ comes from an element 
of $ {\phi}_{i}^{{\cal W}} (k_{{\cal W}})_{{\pi}_{\cal P}} .$ Assume $c_{i,j}\ge 1.$ We have
\begin{equation}\label{ord1} 
{\sigma}({\pi}_{\cal P}^{c_{i,j}-1}{\text{\rm red}}_{{\cal W}_{1}}(y_{i,j}))=(1+{\pi}_{\cal P}^ku){\pi}_{\cal P}^{c_{i,j}-1}{\text{\rm red}}_{{\cal W}_{1}}(y_{i,j}).
\end{equation}
This is because the ${\pi}_{\cal P}$-order of ${\pi}_{\cal P}^{c_{i,j}-1}{\text{\rm red}}_{{\cal W}_{1}}(y_{i,j})$ in $ {\phi}_{i}^{{\cal W}} (k_{{\cal W}})_{{\pi}_{\cal P}} $ is equal to $k+1.$ 
Notice that  we have choosen ${\cal W}$ such that   ${\cal W}\nmid {\cal P}$ and since the reduction map is injective on a torsion prime to ${\cal W}$  (cf. Proposition \ref{injection}) we see, by definition of the field $K_{{\cal P}^{k+1}},$ that there exists a torsion element of the left bottom Mordell-Weil group which maps onto ${\pi}_{\cal P}^{c_{i,j}-1}{\text{\rm red}}_{{\cal W}_{1}}(y_{i,j}).$ Since on this torsion element ${\sigma}$ acts via the homothety $h$ we have (\ref{ord1}).

On the other hand by our choice of ${\cal W}$ the Frobenius at ${\cal W}_1$ acts on ${\pi}_{\cal P}^{c_{i,j}-1}{\text{\rm red}}_{{\cal W}_{1}}(y_{i,j})$ via $\sigma$. Thus since ${\text{\rm red}}_{{\cal W}_{1}}(y_{i,j})\in  {\phi}_{i}^{{\cal W}_2} (k_{{\cal W}_2})_{{\pi}_{\cal P}}$   we have
\begin{equation}\label{ord2}
{\sigma}({\pi}_{\cal P}^{c_{i,j}-1}{\text{\rm red}}_{{\cal W}_{1}}(y_{i,j}))={\pi}_{\cal P}^{c_{i,j}-1}{\text{\rm red}}_{{\cal W}_{1}}(y_{i,j}).
\end{equation}
Comparing (\ref{ord1}) and (\ref{ord2})  we obtain
$${\pi}_{\cal P}^{c_{i,j}-1+k}u\,{\text{\rm red}}_{{\cal W}_{1}}(y_{i,j})={\pi}_{\cal P}^{c_{i,j}-1}u\,{\text{\rm red}}_{{\cal W}_{1}}(x_{i,j})=0.$$
But this contradicts the assumption that the ${\pi}_{\cal P}$-order of  ${\text{\rm red}}_{\cal W}(x_{i,j})$ equals $c_{i,j}.$
\end{proof}

We also have the following

\begin{thm}\label{redthm11}
Let $A={\mathbb F}_q[t]$, ${\phi}_i$ for $1\leq i \leq t$ be  Drinfeld modules of generic characteristic defined over $K$ such that ${\End}_{K^{\mathrm{sep}}}({\phi}_{i})=A.$
Let  ${\cal P}\in M_{A},$  $m\in {\mathbb N}\cup \{0\}$.
Let $x_{i,j}\in {\phi}_{i}({\cal O}_K)$ for some $s_{i}$ and  $1\leq j \leq s_{i}$ be linearly independent elements over $A$  and let $T_{i,j}\in {\phi}_{i}[{\cal P}^m]$  be arbitrary torsion points for all $1\leq j\leq s_{i}$ and $1\leq i\leq t.$     
Then there is a set $W$ of prime ideals $\cal W$ of ${\cal O}_{K}$  of positive density such that  
$${\mathrm{red}}_{{\cal W}^{\prime}}(T_{i,j})={\mathrm{red}}_{{\cal W}}(x_{i,j})$$ in $ {\phi}_{i}^{{\mathcal W}}({\mathcal O}_K/{{\cal W}})_{{\pi}_{\cal P}}$ for all $1\leq j\leq s_{i}$ and  $1\leq i\leq t,$ where 
${\cal W}^{\prime}$ is a prime in $O_L$ over ${\cal W},$ where $L$  is the compositum   of fields ${K({\phi}_{i}[{\cal P}^m])}\, {\mathrm{for}} \, 1\leq i\leq t$, ${\mathrm{red}}_{{\cal W}^{\prime}}:{\phi}_{i}({\cal O_L})\rightarrow {\phi}_{i}^{\cal W}({k_{{\cal W}^{\prime}}})$ and ${k_{{\cal W}^{\prime}}}={{{\cal O}_L}}/{{\cal W}^{\prime}}.$
\end{thm}

\begin{proof}
Put $x_{i,j}=P_{i,j}-T_{i,j},$ take $L$ instead of $K$ and apply Theorem \ref{redthm}.
\end{proof}

\section{Local to global principle for ${\mathbf t}$-modules associated to direct sums of Drinfeld modules}

In what follows we assume to be working in the category of Anderson ${\mathbf t}$-modules that are direct products of Drinfeld modules.
Such situation was considered in \cite{pt06}.
So, we assume that the ${\mathbf t}$-module ${\widehat\varphi}={\phi}_{1}^{e_{1}}\times\dots\times {\phi}_{t}^{{e}_{t}},$ where ${\phi}_{i}$ is a Drinfeld module of generic characteristic of rank $d_{i}. $ We assume that for every $1\leq i \leq t,$  ${\End}_{K^{\mathrm{sep}}}({\phi}_{i})=A$
and all modules are defined over the same ring of integers ${\cal O}_{K}.$ We further assume that we are given a finitely generated $A$-submodule ${N}={N}_{1}^{e_1}\times\dots\times {N}_{t}^{e_t}$ of the Mordell-Weil group ${\widehat\varphi}({\cal O}_{K})={\phi}_{1}({\cal O}_{K})^{e_{1}}\times\dots\times {\phi}_{t}({\cal O}_{K})^{{e}_{t}},$ where $N_i\subset {\phi}_i({\cal O}_K).$
\begin{rem}\label{infgen}
According to the result of B.Poonen \cite{p95} the Mordell-Weil group is a direct sum of a free $A$-module on ${\aleph}_0$ generators and a finite torsion module.
\end{rem}
In view of  Remark \ref{infgen}, the  theorem analogous to Theorem 4.1 of \cite{bk11}  asserts the following:

\begin{thm}\label{localglobal}
Let ${\widehat\varphi}={\phi}_{1}^{e_{1}}\times\dots\times {\phi}_{t}^{e_{t}}$ be a ${\mathbf t}$-module where ${\phi}_{i}$ for $1\leq i \leq t$ are pairwise non-isogenous Drinfeld modules of generic characteristic  defined over ${\cal O}_K.$ 
 We assume that for each  $1\leq i\leq t,$
${\End}_{K^{\mathrm{sep}}}({\phi}_{i})=A.$ Let $N_{i}\subset {\phi}_{i}({\cal O}_{K})$ be a finitely generated $A$-submodule of the Mordell-Weil group. Pick ${\Lambda}\subset N=N_{1}^{e_{1}}\times \dots \times N_{k}^{e_{t}}$ to be an $A$-submodule. Assume that $d_{i}={{\mathrm rank}}({\phi}_{i})\geq e_{i}$ for each $1\leq i\leq t.$ Let $P\in N$ and assume that for almost all primes ${\cal P}$ of ${\cal O}_K$ we have
${\mathrm{red}}_{\cal P}(P)\in {\mathrm{red}}_{\cal P}({\Lambda}).$ Then $P\in {\Lambda}+N_{\tor}.$
\end{thm}

Let $F={\mathbb F}_{q}(t)$ be the field of fractions of $A.$
Notice that $d_{i}={\dim}_{F_{\cal P}}V_{\cal P}({\phi}_{i})$ for $ {\cal P}\in M_A$ (cf.  Remark \ref{dim}).
\begin{cor}\label{incl}
If $N_{\tor}\subset {\Lambda}$ then the following conditions are equivalent
\begin{enumerate}
\item[(1)] \quad $P\in{\Lambda}$
\item[(2)] \quad ${\mathrm{red}}_{\cal P}(P)\in {\mathrm{red}}_{\cal P}({\Lambda})$ for almost all primes of ${\cal O}_K.$
\end{enumerate}
\end{cor}
Recall some facts from \cite{bk11} section 3 concerning modules over division algebras adapted for our situation. Let $D_{i}= F={\End}({\phi}_{i}){\otimes}_{A}F.$
Then ${\End}({\widehat\varphi})=M_{e_1}(A)\times \dots \times M_{e_t}(A)$ and ${\End}({\widehat\varphi}){\otimes}_{A}F= M_{e_1}(F)\times \dots \times M_{e_t}(F).$
\begin{dfn}\label{simple1}
Let $K_{1}(j)$ be the left ideal of the matrix algebra $M_{e_j}(F)$ consisting of matrices ${\bm{\alpha}}(j)_{1}=(a(j)_{l,m}), \quad  1\leq l,m \leq e_j$ such that $a(j)_{l,m}=0$ if $m\neq 1.$
\end{dfn}
\begin{dfn}\label{MeD}
For a $D_{i}$-vector space $W_{i}$ let ${\bm{\omega}}(i)=({\omega}(i),0,\dots ,0)^T\in W_{i}^{e_i}, \quad  {\omega}(i) \in W_{i}.$ Let ${\mathbb D}={\prod}_{i=1}^{t}D_{i},$ ${\mathbb M}_{e}({{\mathbb D}})={\prod}_{i=1}^{t}M_{e_i}(D_{i}),$ where $e=(e_{1},\dots , e_{t}).$
\end{dfn}
\begin{rem}\label{W}
Let $W_{i}$ be a finite dimensional $F$-vector space over $D_{i}, \, 1\leq i\leq t.$ Then $W={\bigoplus}_{i=1}^{t}W_{i}^{e_{i}}$ has an obvious ${\mathbb M}_{e}({\mathbb D})$-module structure.
\end{rem}
The following lemma is essentially a specialisation of Corollary 3.2 of \cite{bk11}.
\begin{lem}\label{simple2}
Every nonzero simple ${\mathbb M}_{e}({\mathbb D})$-submodule of  $W={\bigoplus}_{i=1}^{t}W_{i}^{e_{i}}$ is of the form 
$$K(j)_{1}{\bm{\omega}}(j)=\{\,(a_{1,1}{\omega}({j}),\dots , a_{e_j, 1}{\omega}(j))^T : a_{k,1}\in D_{j}=F, \, 1\leq t\leq e_{j} \}$$
where $1\leq j\leq t$ and ${\omega}(j)\in W_j.$
\end{lem}
The trace homomorphism ${\tr}: {\mathbb M}_{e}(\mathbb D)\rightarrow F$ is defined as ${\tr}={\sum}_{i=1}^t {\tr}_{i}$ where ${\tr}_{i}: M_{e_{i}}(D_{i}) \rightarrow F$ is the usual trace homomorphism.
The following lemma corresponds to  Lemma 3.3 of \cite{bk11}.
\begin{lem}\label{trace}
The induced map
${\tr}: {\Hom}_{{\mathbb M}_{e}({\mathbb D})}(W, {{\mathbb M}_{e}({\mathbb D})})\rightarrow {\Hom}_{F}(W,F)$
is an isomorphism.
\end{lem}
\begin{proof}
Replace ${\mathbb Q}$ by $F$ in the proof of Lemma 3.3 of \cite{bk11}.
\end{proof}
Semisimplicity of ${{\mathbb M}_{e}({\mathbb D})}$ implies that the module $W$ is semisimple and therefore for any $\bm{\pi}\in  {\Hom}_{{\mathbb M}_{e}({\mathbb D})}(W, {{\mathbb M}_{e}({\mathbb D})})$ there exists 
${\bm s}: {\Im}{\bm{\pi}}\rightarrow W$ such that ${\bm{\pi}}\circ {\bm s}={\Id}.$ One has of course the following splittings ${\bm{\pi}}={\prod}_{i=1}^t {\Im}{\bm{\pi}}(i)$ and ${\bm s}={\bigoplus}_{i=1}^t{\bm s}(i)$ where 
$\bm{\pi}(i)\in  {\Hom}_{{M}_{e_i}({D}_i)}(W_i^{e_{i}}, {{M}_{e_i}({D}_i)}),$ $\bm{s}(i)\in  {\Hom}_{{M}_{e_i}({D}_i)}({\Im}{\bm{\pi}}(i), W_i^{e_i})$ and ${\bm{\pi}}(i)\circ {\bm s}(i)={\Id}.$

\subsection{Proof of Theorem \ref{localglobal}}
\begin{proof}

Since the $A$-torsion of $N$ is finite (cf. Remark \ref{infgen} ) we can consider a torsion free $A$-module 
${\Omega}:= cN, $  
   where $c=g_{1}\cdot ... \cdot g_{k}$ is the product of  generators of the $A$-anihilator of $N_{\tor}$, 
and replace $N$ by ${\Omega}.$ We can also assume that ${\Lambda}\subset {\Omega}$ and $P\in {\Omega}.$
Let ${P}_1, \dots {P}_s$ be an $A$-basis of ${\Omega}$ such that 
$${\Lambda}=Av_1P_1+\dots +Av_sP_s, \quad P=n_1P_1+\dots +n_sP_s$$
where $v_{i} ,n_{i}\in A$ for $1\leq i \leq s.$
Assume that $P\notin {\Lambda}.$ This is equivalent to $P{\otimes}1 \notin {\Lambda}{\otimes}_{A}{A}_{{\cal { U}}}$ for some   $\, {\cal U}\in A$ where 
${A}_{{\cal { U}}}$ is the completion of $A$ with respect to ${\cal U}.$
Thus there exists  $1\leq j_0 \leq s$ and a natural number $m_1$ such that ${\cal U}^{m_{1}} || n_{j_0}$ and  ${\cal U}^{m_{1}+1} | v_{j_0}.$
Define the following map of $A$-modules
$$
{\pi}:{\Omega}\rightarrow A, \qquad {\pi}(R)=a_{j_0}, \qquad R={\sum}_{i=1}^sa_iP_{i},\quad a_{i}\in A.
$$
We shall use the same letter ${\pi}$ as a short abbreviation for the map ${\pi}{\otimes}_{A}id_F : {\Omega}{\otimes}_{A}F \rightarrow F.$ By Lemma \ref{trace} we obtain a map ${\bm{\pi}}   \in {\Hom}_{{\mathbb M}_{e}({\mathbb D})}({\Omega}{\otimes}_AF,\, {{\mathbb M}_{e}({\mathbb D})})$ such that ${\tr}{\bm{\pi}}=\pi .$ By the above discussion we also have ${\bm{s}}$ such that ${\bm{\pi}}\circ {\bm{s}}=\Id .$
We have
$${\Omega}{\otimes}_AF\cong {\Im{\bm s}}\oplus {\text{\rm Ker}{\bm\pi}} \,\,\,\, {\text{\rm and}}\,\,\,\,  {\Omega}^{e_i}{\otimes}_AF\cong {\Im{\bm s}(i)}\oplus {\text{\rm Ker}{\bm\pi}(i)},\,\, 1\leq i \leq t.$$
By Lemma \ref{simple2} we have the following decompositions 
 $${\Im{\bm s}(i)}={\bigoplus}_{k=1}^{k_{i}}K(i)_{1}{\bm{\omega}}_{k}(i) \,\,\,\, {\text{\rm and}} \,\,\,\,  {\text{\rm Ker}{\bm\pi}(i)}={\bigoplus}_{k=k_i+1}^{u_{i}}K(i)_{1}{\bm{\omega}}_{k}(i).$$
 By assumptions $k_{i}\leq e_{i}\leq d_{i},$
for every $1\leq i \leq t.$ 
The elements ${\omega}_{1}(i),\dots , {\omega}_{k_i}(i),\dots , {\omega}_{u_i}(i)$ constitute a basis for the $F$-vector space ${\Omega}_i{\otimes}_{A}F.$ (Notice that $D_i=F$ by assumption.)
Without loss of generality one can assume that ${\omega}_{k_{i}+1}(i), \dots ,{\omega}_{u_i}(i) \in {\Omega}_i.$ ${\Omega}_i{\otimes}_{A}F $ is a free $F$-module.
We have ${\cal R}={\End_{A}{\Omega}}\subset {\mathbb M}_{e}({\mathbb D})={\cal R}{\otimes}_AF.$ Since ${\Omega}$ is a finitely generated $A$-module there exists a polynomial 
$M_0\in A$  such that the following homomorphisms are well defined
$M_0{\bm{\pi}} : {\Omega}\rightarrow {\cal R}\quad \,\, {\text{\rm and}}  \,\,\quad {\bm{s}}: M_0{\bm{\pi}} {\Omega}\rightarrow {\Omega}.$
Define the $M_{e_{i}}(A)$-module
$
{\Gamma}(i)={\sum}_{k=1}^{k_{i}}K(i)_1M_{0}{\bm{\omega}}_{k}(i)+{\sum}_{k=k_{i}+1}^{u_{i}}K(i)_{1}{\bm{\omega}}_{k}(i)\subset {\Omega}_{i}^{e_{i}}.
$
Let ${\Gamma}={\bigoplus}{\Gamma}(i)\subset \Omega$, $M_{2},M_{3}\in A$ be  polynomials of minimal degrees such that $M_{2}{\Omega} \subset {\Gamma}$ and $M_{3}{\Gamma} \subset M_{2}{\Omega}.$
The choice of $j_0$ implies ${\pi}(P)\notin {\pi}(\Lambda {\otimes}_{A} {A}_{{\cal U}})+ {\cal U}^m{\pi}({\Omega}{\otimes}_A{A}_{{\cal U}})$
for every $m> m_{1}.$ Choose such an $m.$ 
Then since $\tr M_{0}{\bm{\pi}}=M_{0}{\pi}$ we have 
\begin{equation}\label{uncon}
M_{0}{\bm\pi}(P)\notin M_{0}{\bm\pi}(\Lambda {\otimes}_{A} {A}_{{\cal U}})+ M_{0}\,{{\cal U}}^m{\bm\pi}({\Omega}{\otimes}_A{A}_{{\cal U}}).
\end{equation}
Let $K(i)_{1,{{\cal U}}}=K(i)_{1}{\otimes}_{A}{A}_{{\cal U}} \subset M_{e_{i}}({A}_{{\cal U}})$ and $Q\in \Lambda$.
By the definition of  $M_{2}\in A$ we have
\begin{equation}\label{eqkk2}
M_{2}(P\otimes 1 - Q{\otimes} 1)= M_{0}^2{\sum}_{i=1}^t{\sum}_{k=1}^{k_{i}}({\bm{\alpha}}_k(i)_{1} - {\bm{\beta}}_k(i)_{1} ){\bm{\pi}}( {\bm{\omega}}_{k}(i))
\end{equation}
where ${\bm{\alpha}}_{k}(i)_1, \, {\bm{\beta}}_{k}(i)_1 \in K(i)_{1, \,{{\cal U}}}, \,\, 1\leq k \leq u_i, \,\, 1\leq i\leq t .$
By (\ref{uncon}) and (\ref{eqkk2})  and the choice of $M_3$ we have 
$$M_{0}^2{\sum}_{i=1}^t{\sum}_{k=1}^{k_{i}}({\bm{\alpha}}_k(i)_{1} - {\bm{\beta}}_k(i)_{1}) {\bm{\pi}}( {\bm{\omega}}_{k}(i)) \notin {\cal U}^mM_{0}{\bm{\pi}}(M_{3}{\Gamma}).$$
This implies that for some $1\leq i\leq t$ and $1\leq k \leq k_{i}$ we have
\begin{equation}\label{uncon1}
{\bm{\alpha}}_k(i)_{1} - {\bm{\beta}}_k(i)_{1}\notin {{\cal U}}^m M_{3}K(i)_{1,\,{{\cal U}}}.
\end{equation}

Notice that for all $n^{\prime}\in {\mathbb N}$ there is an isomorphism
${\phi}_{i}[{\cal U}^{n^{\prime}}]\cong T_{\cal U}({\phi}_{i})/{\upsilon}^{n^{\prime}}T_{\cal U}({\phi}_{i})$
where ${\upsilon}$ is a generator of the ideal ${\cal U}.$
Let ${\cal L}_{\cal U}={\bigoplus}_{i=1}^{t}T_{\cal U}({\phi}_{i})$
and ${\eta}_{1}(i),\dots ,{\eta}_{d_{i}}(i)$ be  a basis of $T_{\cal U}({\phi}_{i})$ over $A_{\cal U}$ (cf. Remark \ref{dim}).
Let $m_0$ and $m_{3}$ be the natural numbers such that ${\upsilon}^{m_{0}}||M_{0}$ and ${\upsilon}^{m_{3}}|| M_{3}.$
Let $V_{{\cal U},i}=T_{{\cal U}}({\phi}_i){\otimes}_{A_{{\cal U}}}F_{\cal U}.$ Then ${\dim}_{F_{\cal U}}V_{{\cal U},i}=d_{i}$.
The quotient $T_{\cal U}({\phi}_{i})/{\upsilon}^{n^{\prime}}T_{\cal U}({\phi}_{i})$ of $A_{\cal U}$-modules is a free $A_{\cal U}/({\upsilon}^{n^{\prime}}A_{\cal U})$ - module with the basis 
$T_{1}(i),\dots , T_{d_{i}}(i)$ where $T_{k}(i)$ is the image of ${\eta}_{k}(i)$ in ${\phi}_i[{\cal U}^{n^{\prime}}].$ 
Let $\bm{\eta}_k(i)=({\eta}_k(i),0,\dots , 0)^T \in T_{\cal U}({\phi}_{i})^{e_i}$ and ${\bm T}_{k}(i)=(T_k(i),0,,\dots, 0)^T\in {\phi}_i[{\cal U}^{n^{\prime}}]^{e_i}$ (cf. Definition \ref{MeD} ).
Take $n^{\prime}>m+m_0 +m_3.$ By Theorem \ref{redthm11}, applied for $ {\cal P}= \cal U$, there exists a set of primes ${\cal W}\in{\cal O}_L$  (where $L$ is the compositum of fields defined  in Theorem \ref{redthm11}) of positive density such that
\begin{equation}\label{redut}
{\text{\rm red}}_{\cal W}({\omega}_{k}(i))=0, \quad {\mathrm{for}}\quad 1\leq i\leq t,\,\, k_{i}+1\leq k\leq u_{i}
\end{equation}
and
\begin{equation}\label{redut1}
{\text{\rm red}}_{\cal W}({\omega}_{k}(i))={\text{\rm red}}_{\cal W}(T_{k}(i)), \quad {\mathrm{for}}\quad 1\leq i\leq t,\,\, 1\leq k\leq k_{i}.
\end{equation}
Pick such a prime ${\cal W}$ that does not divide ${\cal U}$. Since by assumption ${\text{\rm red}}_{\cal W}(P)\in {\text{\rm red}}_{\cal W}({\Lambda})$ we can choose $Q\in\Lambda$ such that ${\text{\rm red}}_{\cal W}(P)={\text{\rm red}}_{\cal W}(Q).$
Now we apply the reduction map ${\text{\rm red}}_{\cal W}$ to the equation
$$
M_2(P-Q)  ={\sum}_{i=1}^{t}{\sum}_{k=1}^{k_{i}}({\bm{\alpha}}_k(i)_{1} - {\bm{\beta}}_k(i)_{1})M_{0}{\bm{\omega}}_{k}(i)+
   {\sum}_{i=1}^{t}{\sum}_{k=k_{i}+1}^{u_{i}}({\bm{\alpha}}_k(i)_{1} - {\bm{\beta}}_k(i)_{1}){\bm{\omega}}_{k}(i).
$$

Thus we obtain 
$
0={\sum}_{i=1}^{t}{\sum}_{k=1}^{k_{i}}({\bm{\alpha}}_k(i)_{1} - {\bm{\beta}}_k(i)_{1})M_{0}{\text{\rm red}}_{\cal W}({\bm{T}}_{k}(i)).
$
  \, Since the reduction map ${\text{\rm red}}_{\cal W}$ is injective on a torsion prime  to a characteristic of a Drinfeld module, cf. Proposition \ref{injection}, we have
$0={\sum}_{i=1}^{t}{\sum}_{k=1}^{k_{i}}({\bm{\alpha}}_k(i)_{1} - {\bm{\beta}}_k(i)_{1})M_{0}{\bm{T}}_{k}(i).$
Therefore the element \linebreak
${\upsilon}^{m_{0}}{\sum}_{i=1}^{t}{\sum}_{k=1}^{k_{i}}({\bm{\alpha}}_k(i)_{1} - {\bm{\beta}}_k(i)_{1}){\bm{\eta}}_{k}(i)$
maps to zero in $A_{\cal U}/{\upsilon}^{n^{\prime}}A_{\cal U}$-module ${\cal L}_{\cal U}/{\upsilon}^{n^{\prime}}{\cal L}_{\cal U}.$
This yields 
$${\sum}_{i=1}^{t}{\sum}_{k=1}^{k_{i}}({\bm{\alpha}}_k(i)_{1} - {\bm{\beta}}_k(i)_{1}){\bm{\eta}}_{k}(i)\in {\upsilon}^{n^{\prime}-m_0}{\cal L}_{\cal U}.$$
But since ${\eta}_{k}(i) ,  1\leq k\leq d_{i},$ constitute a basis of $T_{\cal U}({\phi}_{i})$ over $A_{\cal U}$ we obtain the following:
\begin{equation}\label{eqk5} 
{\bm{\alpha}}_k(i)_{1} - {\bm{\beta}}_k(i)_{1} \in {\upsilon}^{n^{\prime}-m_{0}}K(i)_{{1}, \,{\cal U}}.
\end{equation}
But (\ref{eqk5}) contradicts (\ref{uncon1}).
\end{proof}
We also have the following theorem:
\begin{thm}\label{best}
Let ${\phi}$ be a Drinfeld module of rank $d$ defined over ${\cal O}_{K}$  with ${\End}(\phi)=A.$ 
Then the numerical bound in Theorem \ref{localglobal} is the best possible. That is, 
 for the ${\mathbf t}$-module ${\phi}^{d+1}$ the local to global principle of Theorem \ref{localglobal} does not hold.
\end{thm}
\begin{proof}
Our proof is modelled on the counterexample to the local  to global principle for abelian varieties constructed by P. Jossen and A.Perucca in \cite{jp10}.
Let $e=d+1$ and $P_{1},\dots ,P_{e}\in {\phi}({\cal O}_{K})$ be points linearly independent over $A.$
Let
\begin{align*}\label{cex}
    P &:= \begin{bmatrix}
           P_{1} \\
           P_{2} \\
           \vdots \\
           P_{e}
         \end{bmatrix}
         \quad
         {\Lambda}:= \{ MP \,: \quad M\in {\mathrm{Mat}}_{e\times e}(A), \quad {\tr}M=0 \}.
  \end{align*}
Denote for shorthand  ${\kappa}={{\cal O}_{K}}/{\cal P}.$ Notice that $P\notin {\Lambda}$ since $P_{1},\dots ,P_{e}$ are $A$-linearly independent.
Let ${\cal W}\in {\cal O}_{K}$ be a prime of good reduction for ${\phi}.$
We will find a matrix $M\in {\mathrm{Mat}}_{e\times e}(A)$ such that $\overline{P}=M{\overline{P}}$ where $\overline{P}=[\overline{P}_{1}\,,\dots ,\overline{P}_{e}]^{T}$ is the reduction  ${\mod}\,\, {\cal W}$ of the point $P.$ This will show that ${\mathrm{red}}_{\cal W}P\in {\mathrm{red}}_{\cal W}{\Lambda}.$
Since the Mordell-Weil group ${\phi}({\kappa})[\cal P]$ is finite there exist polynomials ${\alpha}_{1},\dots , {\alpha}_{e}\in A$ of minimal degrees such that
\begin{align*} 
{\alpha}_{1}{\overline P}_{1} +   m_{1,2}{\overline P}_2 + \dots  +  m_{1,e}{\overline P}_{e}  &=  0 \\
{m_{2,1}}{\overline P}_{1} +   {\alpha}_{2}{\overline P}_2 + \dots   +  m_{2,e}{\overline P}_{e}  &=  0 \\
\dots\dots\dots\dots\dots\dots\dots\dots\dots   &    \\
{m}_{e,1}{\overline P}_{1} +   m_{e,2}{\overline P}_2 + \dots  +  {\alpha}_{e}{\overline P}_{e}  &=  0 .
\end{align*}
We will show that $D=\gcd({\alpha}_{1},\dots , {\alpha}_{e})=1.$ 
It is enough to show that for any prime ideal ${\cal P}\lhd A$ the polynomial  $ D$ is not divisible by ${\cal P}.$
Assume opposite   and  let $\tilde{\cal P}$ be a prime that divides $D.$  This means, by our choice of ${\alpha}_{1},\dots , {\alpha}_e$  that ${\tilde{\cal P}}$  divides  coefficients of any linear combination of points ${\bar P}_{1},\dots, {\bar P}_e \in {\phi}^{\tilde{\cal P}}.$  
By Proposition \ref{wang} the group ${\phi}^{\tilde{\cal P}}({\kappa})[{\cal P}]$ is isomorphic to $(A/{\tilde{\cal P}})^{d-{\bar h}}.$
Therefore the group $Y=( {\bar P}_{1},\dots, {\bar P}_e )\cap {\phi}^{\tilde{\cal P}}({\kappa})[{\cal P}]$ is generated by fewer than $e$ elements.
We may assume without loss of generality that $Y=( {\bar P}_{2},\dots, {\bar P}_e )\cap {\phi}^{\tilde{\cal P}}({\kappa})[{\cal P}].$
Let 
\begin{equation}\label{cexd1}
{\alpha}_{1}{\bar P}_{1}+ {x}_{2}{\bar P}_{2}+\dots +{x}_{e}{\bar{P_{e}}}=0
\end{equation}
be a linear relation. Then since the left hand side of (\ref{cexd1}) is in $Y$ we obtain a contradiction with the minimality of ${\alpha}_{1}.$
Thus $D=1.$ 
Hence there exist $a_{1},\dots ,a_{e}\in A$ such that 
$$e=a_{1}{\alpha}_{1}+\dots  +a_{e}{\alpha}_{e}. $$
 Put $m_{i,i}=1-a_{i}{\alpha}_{i }.$  Then $m_{1,1}+\dots +m_{e,e} =0$ and 

\begin{align*}
\begin{bmatrix}
m_{1,1} & \dots & m_{1,e}\\
\dots &  \dots & \dots \\
m_{e,1} & \dots & m_{e,e}
\end{bmatrix}
\begin{bmatrix}
{\bar P}_{1}\\
 \dots \\
{\bar P}_{e}
\end{bmatrix}
=
\begin{bmatrix}
{\bar P}_{1}\\
 \dots \\
{\bar P}_{e}
\end{bmatrix} .
\end{align*}

Therefore ${\bar P}\in {\bar{\Lambda}}.$        
\end{proof}
Essentially the same proof - with ${\mathbb Z}$ substituted for $A$ - works for abelian varieties with ${\End}{\mathbb A}={\mathbb Z}$ and we obtain the following:

\begin{thm}\label{best1}
Let ${\mathbb A}/F$ be an abelian variety defined over a number field $F$  with ${\End}{\mathbb A}=\mathbb Z.$ Let $d={\dim}_{{\End}_{F^{\prime}}({\mathbb A}_i)^{0}}\,
H_{1} ({\mathbb A}({\mathbb C});\, {\mathbb Q})=2g,$ where $g={\dim}{\mathbb A}.$ Assume $ {\mathrm{rank}}A(F)>d.$ 
Then the numerical bound in Theorem \ref{thm A} is the best possible. That is,  for  ${\mathbb A}^{d+1}$ the local to global principle of Theorem \ref{thm A} does not hold.
\end{thm}

\section{Some other consequences of the reduction theorem}

In \cite{b17} S.Bara{\'n}czuk introduced a dynamical version of the local to global principle \ref{question}. In \cite{b17} the author considers  the case of abelian groups (modules over ${\mathbb Z}$ ) satisfying the following two axioms:
\begin{ass}\label{ass1}
Let $B$ be an abelian group such that there are homomorphisms ${\mathrm{red}}_{v}: B\rightarrow B_{v}$ for an infinite family $v\in F$, whose targets $B_v$ are finite abelian groups.
\begin{enumerate}
\item[(1)] Let $l$ be a prime number, and $(k_{1},\dots , k_m)$ be  a sequence of nonnegative integers. If $P_1,\dots , P_m \in B$ are points linearly independent over ${\mathbb Z},$ then there is a family of primes $v$ in $F$ such that 
$l^{k_{i}}|| {\ord}_{v}P_{i}$ if $k_{i}>0$ and $l\nmid {\ord}_{v}P_i$ if $k_i=0.$
\item[(2)] For almost all $v$ the map $B_{\text{\rm tors}}\rightarrow B_{v}$ is injective.
\end{enumerate}
\end{ass}
Here ${\ord}_{v}P$ is the order of a reduced point $P\mod v$.
In \cite{b17} and \cite{ba17} finitely generated abelian groups are considered. In our case we modify Assumption \ref{ass1} (2) so that we deal appropriately with infinite torsion in Drinfeld modules.
Instead of Assumption 6.1 (1) we assume  the following:
\begin{ass}\label{ass2}
Let $B$ be an $A$-module such that there are homomorphisms ${\mathrm{red}}_{\cal U}: B \rightarrow B_{\cal U}$ for an infinite family ${\cal U}\in SpecO_K$ such that ${ B}_{\cal U}$ is a torsion $A$-module.
For $P\in B$ let ${\ord}_{\cal U}P$ be the polynomial of minimal degree in ${\cal O}_K$ such that it annihilates ${\mathrm{red}}_{\cal U}P$
\begin{enumerate}
  \item[(1)] For  ${\cal U}$   and $(k_{1},\dots , k_m),$  a sequence of nonnegative integers, the following holds true: If $P_1,\dots , P_m \in B$ are points linearly independent over ${A},$ then there is a family of primes ${\cal W}$ in ${\cal O}_K$ such that 
${\cal U}^{k_{i}}|| {\ord}_{\cal W}P_{i}$ if $k_{i}>0$ and ${\cal U}\nmid {\ord}_{\cal W}P_i$ if $k_i=0.$

\item[(2)] For any ${\cal U}$ there are infinitely many primes ${\cal W}$ in ${\cal O}_K$ such that the reduction map is an injection on ${\cal U}$ - torsion, i.e. ${\mathrm{red}} _{\cal W}: B[\cal U]\hookrightarrow B_{\cal W}[\cal U].$
\end{enumerate}
\end{ass}  
The main theorem of \cite{b17} can be extended to the case of ${\mathbf t}$-modules considered in this paper in the following way.
\begin{thm}\label{stefbar1}
Let ${\widehat\varphi}={\phi}_{1}^{e_{1}}\times\dots\times {\phi}_{t}^{e_{t}}$ be a ${\mathbf t}$-module where ${\phi}_{i},\, 1\leq i \leq t$ are pairwise non-isogenous Drinfeld modules defined over ${\cal O}_K$  such that 
${\End}{\phi}_{i}=A.$  Let  ${\Lambda}\subset {\phi}({\cal O}_K)$  be a finitely generated $A$-submodule. 
 For $ x\in {\phi}({\cal O}_{K})$ and $ w(t)\in A$ let $O_{w(t)}(x)=\{ w^{n}(x): n\geq 0 \}$ be an orbit of the point $x$  under the iterations of   multiplication by $w(t)\in A.$
Then the following are equivalent
\begin{enumerate}
\item[(1)] For almost every $\cal U \in Spec(O_K)$
$$O_{w(t)}({\mathrm{red}}_{\cal U}(P))\cap {\mathrm{red}}_{\cal U}({\Lambda}) \neq \emptyset$$
\item[(2)] $O_{w(t)}(P)\cap {\Lambda}\neq\emptyset$
\end{enumerate}
\end{thm}
whereas the main theorem of \cite{ba17} extended to ${\mathbf t}$-modules reads as follows:
\begin{thm}\label{stefbar2}
Let ${\widehat\varphi}={\phi}_{1}^{e_{1}}\times\dots\times {\phi}_{t}^{e_{t}}$ be a ${\mathbf t}$-module where ${\phi}_{i},\, 1\leq i \leq t$ are pairwise non-isogenous Drinfeld modules defined over ${\cal O}_K$  such that 
${\End}{\phi}_{i}=A.$ Let $P,Q\in {\phi}({\cal O}_{K})$ and $w_{1}(t), w_{2}(t)\in A.$ Suppose that for almost all ${\cal U}\in Spec{\cal O}_K$ there exists a natural number $n_{\cal U}$ such that 
$$ {\mathrm{red}}_{\cal U}(w_{1}^{n_{\cal U}}P-w_{2}^{n_{\cal U}}Q)=0$$
Then there exists a natural number $n$ and a torsion point $T\in {\phi}(O_K)$ of an order that divides some power of ${\gcd}(w_{1},w_{2})$ such that $w_{1}^nP-w_{2}^nQ=T.$
\end{thm}
Proofs of Theorems \ref{stefbar1} and \ref{stefbar2} follow the lines of   S.Bara{\'n}czuk's original proofs. In appropriate places one has to replace multiplication by a natural number (viewed as an element of ${\End}B$ ) by the 
$A$-module action of an element $w(t)\in A$  and Assumption \ref{ass1} by \ref{ass2}.  Additionally one has to check that Assumptions \ref{ass2} is fulfilled. Assumption \ref{ass1} (2) is fulfilled by $B={\phi}({\cal O}_{K})$ cf. (2) of Proposition \ref{injection}. As for  Assumption \ref{ass1} (1) we readily see that it follows from Theorem \ref{redthm1}.

{}

\end{document}